\documentclass[12pt,leqno]{amsart}
\usepackage{amsmath,amsfonts,amssymb,amsthm}
\usepackage{graphicx}
\graphicspath{ }
\usepackage{amsmath}
\usepackage{csquotes}
\usepackage{wrapfig}
\usepackage{accents}
\usepackage{caption}
\usepackage{subcaption}
\usepackage{calligra}
\usepackage[colorlinks]{hyperref}
\hypersetup{citecolor=black}

\numberwithin{figure}{section}
\numberwithin{equation}{section}

\theoremstyle{plain}
\newtheorem{thm}{Theorem}[section]

\theoremstyle{definition}

\theoremstyle{remark}

\usepackage{mathtools}
\title{On Rational Invariant Summation of p-Adic Power Series with Binomial Coefficients}
\author[A. A. Shaikh]{Absos Ali Shaikh$^{1}$}
\address{Department of Mathematics,\\ The University of Burdwan,\\ \newline Burdwan-713101, West Bengal, India.}
\email{$^1$aask2003@yahoo.co.in, aashaikh@math.buruniv.ac.in}
\author[M. A. Sarkar]{MABUD ALI SARKAR$^2$}
\address{Department of Mathematics\\ The University of Burdwan \\ Burdwan-713101, India.}
\email{$^2$mabudji@gmail.com}

\begin{document} 
\maketitle
\begin{abstract}
	In the work we have considered p-adic functional series with binomial coefficients and discussed its p-adic convergence. Then we  have derived a recurrence relation following with a  summation formula which is invariant for rational argument. More precisely, we have investigated a sufficient condition under which the p-adic power series converges to a rational invariant sum $0$. Finally we have shown application of the invariant summation formula to get some intersting relations involving Bernoulli numbers and Bernoulli polynomials. 
\end{abstract}
\footnotetext{
	$\mathbf{2010}$\hspace{10pt}Mathematics\; Subject\; Classification: 12J25, 32P05, 26E30, 40A05, 40D99, 40A30.}
\section{\textbf{Introduction}}
During the last three decades, the p-adic number theory played an essential role in number theory due to its applications in various physical phenomenon (see \cite{B4}) as well as to solve many important problems in various fields of mathematics, especially, on number theory (see \cite{M-L},\cite{PM}). Again, the investigation of convergence of series plays an important role in the theory of mathematical analysis. 

The infinite series of rational numbers can be considered both in p-adic and real number field as rational numbers are endowed with both real norm and p-adic norm. The most interesting fact about the real number series which diverges in real field, may converge in p-adic field and hence needs p-adic investigation. It is important to investigate p-adic series having rational sum for rational argument because physical measurement are considered as rational values. Recently, Dragovich (\cite{B1},\cite{B2},\cite{B3}) introduced p-adic invariant summation of a class of infinite p-adic functional series $\sum n!P_k(n,x)x^n$, $\sum \epsilon^n (n+v)!P_{k \alpha}(n,x)x^{\alpha n+\beta}$, $\sum \epsilon^n n!P_k^{\epsilon}(n,x)x^n$ each containing factorial coefficients $n!$ such that for any rational argument the corresponding sum is also a rational number, where $P_k(n,x)$ are polynomials in $x$ of degree $k$, $\epsilon=\pm1$, $k,v, \beta \in \mathbb{N} \cup \{0\}$, $\alpha \in \mathbb{N}$. But the case when the above p-adic power series contains binomial coefficients instead of factorials coefficients $n!$, has not been covered so far.
In this paper we have considered p-adic functional series with \textit{binomial coefficients} instead of factorial $n!$ and we have arrived with a summation formula having invariant form which gives rational number for rational argument. 
The outline of the paper is as follows: Section 2 is considered with the main results (see, Theorem $\eqref{thm1}$ and Theorem $\eqref{thm2}$). The last section is an application part which deals with a connection of the obtained invariant summation formula $\eqref{26}$ with Bernoulli numbers and Bernoulli polynomials. 
\section{\textbf{Main Results}}
We consider the following p-adic functional series
\begin{equation} \label{eq1}
S_k(x)=\sum_{n=0}^{\infty} \binom{2n}{n} P_{k}(n,x)x^{n}=P_k(0,x)+\binom{2}{1}P_k(1,x)x+\binom{4}{2}P_k(2,x)x^2+\cdots,
\end{equation} 	
where \begin{eqnarray}
P_k(n,x)=B_k(n)x^{k}+B_{k-1}(n)x^{k-1}+\cdots+B_1(n)x+B_0(n), \\
k \in \mathbb{N}_0=\mathbb{N} \cup \{0\} \ \text{and} \ x \in \mathbb{Q}_p. \nonumber
\end{eqnarray} 
Here $B_k(n), \ 0 \leq l \leq k$, represents polynomials in $n$ having integer coefficients. Our goal is to find the polynomials $P_k(n,x)$ for which the series $ \eqref{eq1}$ converges p-adically and has invariant sum. In particular, for rational argument $x \in \mathbb{Q}$, the sum $S_k(x) \in \mathbb{Q}$. 

We note that the power series $\sum a_n x^n$ converges p-adically if and only if $|a_nx^n|_p \to 0$ as $n \to \infty$. Since the p-adic functional series $\eqref{eq1}$ contains the binomial coefficients $\binom{2n}{n}$, we must know the p-adic valuation of $\binom{2n}{n}$ in order to check p-adic convergence of it. In 1852, Kummer (\cite{E1}) showed a method to calculate p-adic valuation of binomial coefficients given as follows:
\begin{thm}(Kummer \cite{E1}) 
	Given integers $n \geq m \geq 0$ and a prime number $p$, the p-adic valuation $\binom{n}{m}$ is equal to the number of carries when $m$ is added with $n-m$ in base $p$. 
\end{thm} 
For $n=n_0+n_1p+\cdots+n_rp^r$, the sum of digits is denoted as $\delta_p(n)$.
By Legendre's formula we have 
\begin{equation}
v_p(n!)=\sum_{k=1}^{r}\left \lfloor \frac{n}{p^k} \right \rfloor=\frac{n-\delta_p(n)}{p-1} \ \text{and} \ |n!|_p=p^{-\frac{n-\delta_p(n)}{p-1}}.
\end{equation}
Then we have
\begin{align}
v_p \left(\binom{n}{m} \right) &=\sum_{k=0}^{r} \left(\left \lfloor \frac{n}{p^k} \right \rfloor -\left \lfloor \frac{m}{p^k} \right \rfloor-\left \lfloor \frac{n-m}{p^k} \right \rfloor \right) \\
&=\frac{\delta_p(m)+\delta_p(n-m)-\delta_p(n)}{p-1}, \nonumber
\end{align}
where each of the summands $\left \lfloor \frac{n}{p^k} \right \rfloor, \left \lfloor \frac{m}{p^k} \right \rfloor, \left \lfloor \frac{n-m}{p^k} \right \rfloor$ is $ \leq 1$.
 Therefore, 
\begin{equation}
\left|\binom{n}{m} \right|_p=p^{-\frac{\delta_p(m)+\delta_p(n-m)-\delta_p(n)}{p-1}}.
\end{equation} 

For $x=1$, the p-adic norm of the general term in $\eqref{eq1}$ is given by
\begin{align*}
	\left|\binom{2n}{n} P_{k}(n,x)x^n \right|_p \leq \left|\binom{2n}{n}x^n \right|_p &=p^{- \frac{2\delta_p(n)-\delta_p(2n)}{p-1}}|x^n|_p
	\\ &=\left(p^{- \frac{2\delta_p(n)-\delta_p(2n)}{n(p-1)}}|x|_p \right)^n.
\end{align*}
Thus the power series $\sum \binom{2n}{n} P_{k}(n,1)x^n$ converges for $x|_p \leq 1$. 
Since $|P_{k}(n,x)|_p \leq |P_{k}(n,1)|_p$ for $|x|_p \leq 1$, the power series $\eqref{eq1}$ converges p-adically on the set of rationals $$ \{x\in \mathbb{Q}: x=\frac{u}{v}, u,v \in \mathbb{Z}, (u,v)=1, \ |u|_{\infty} \leq |v|_{\infty}, \ p|u, p \nmid v  \}.$$ 
\newline
Now we  prove the following theorem:
\begin{thm} \label{thm1}
	Let 
	\begin{equation}
	A_k(n,x)=b_k(n)x^k+b_{k-1}(n)x^{k-1}+\cdots+b_1(n)x+b_0(n)
	\end{equation}
	be a polynomial with polynomial coefficients $b_j(n), \ 0 \leq j \leq k$ in $n$ with rational coefficients. Then there exists such a polynomial $A_{k-1}(N,x), \ N \in \mathbb{N}$ satisfying the equality 
	\begin{equation}
	\sum_{n=0}^{N-1}  \binom{2n}{n} \left[n^{k}(4x-1)^k+U_k(x) \right]x^n= \binom{2N}{N}x^{N}A_{k-1}(N,x),   
	\end{equation}
	where
	\begin{equation}
	U_k(x)=2xA_{k-1}(1,x)-A_{k-1}(0,x), \ k \geq 1, \ x \in \mathbb{Q} \setminus \left\{\frac{1}{4} \right\}.
	\end{equation}
\end{thm}
\begin{proof}
Let us consider the finite power series with binomial coefficients:
\begin{equation} \label{eq2}
S_{k}(N,x)=\sum_{n=0}^{N-1} \binom{2n}{n} n^{k} x^n; \ k,n \in \mathbb{N} \cup \{0\}
\end{equation}	
Putting $k=0$ in equation $\eqref{eq2}$, we have
 \begin{align*}
 	S_0(N,x) &=\sum_{n=0}^{N-1} \binom{2n}{n} x^n, \\
 	&=1+\sum_{n=1}^{N-1} \binom{2n}{n} x^n.
 \end{align*}
\newpage
 We have used the following combinatorics relations:
\begin{eqnarray*}
	 \binom{n+1}{r}&=&\binom{n}{r}+\binom{n}{r-1}, \\
	 \binom{n}{r} &=&\frac{n}{r} \binom{n-1}{r-1}, \\
	 \binom{n}{r} &=&\frac{n-r+1}{r} \binom{n}{r-1}.
	\end{eqnarray*}

Let us derive the following recurrence formula: 
\begin{align*}
	S_k(N,x) &= \sum_{n=0}^{N-1} \binom{2n}{n} n^{k}x^n, \ k \geq 0 \\
	&=2x+ \sum_{n=2}^{N-1} \binom{2n}{n} n^{k}x^n \\
	&=2x+\sum_{n=2}^{N}\binom{2n}{n} n^{k}x^n-\binom{2N}{N} N^{k}x^N \\
	&=2x+ 2 \sum_{n=2}^{N}\binom{2n-1}{n-1} n^{k}x^n-\binom{2N}{N} N^{k}x^N \\
	&=2x+ 2 \sum_{n=1}^{N-1}\binom{2n+1}{n} (n+1)^{k}x^{n+1}-\binom{2N}{N} N^{k}x^N \\
	&=2x+2 \sum_{n=1}^{N-1} \left[\binom{2n}{n}+\binom{2n}{n+1} \right](n+1)^{k}x^{n+1}-\binom{2N}{N} N^{k}x^N \\
	&=2x+2 \sum_{n=1}^{N-1}\binom{2n}{n} (n+1)^{k}x^{n+1}+2
	 \sum_{n=1}^{N-1}\binom{2n}{n+1} (n+1)^{k}x^{n+1}-\binom{2N}{N} N^{k}x^N \\
	&=2x+2 \sum_{n=1}^{N-1}\binom{2n}{n} (n+1)^{k}x^{n+1}+2 \sum_{n=1}^{N-1} \frac{n}{n+1}\binom{2n}{n} (n+1)^{k}x^{n+1}-\binom{2N}{N} N^{k}x^N \\
	&=2x+2 \sum_{n=1}^{N-1}\binom{2n}{n} (n+1)^{k}x^{n+1}+2 \sum_{n=1}^{N-1} \frac{(n+1)-1}{n+1}\binom{2n}{n} (n+1)^{k}x^{n+1}-\binom{2N}{N} N^{k}x^N \\
	&=2x+4 \sum_{n=1}^{N-1}\binom{2n}{n} (n+1)^{k}x^{n+1}-2 \sum_{n=1}^{N-1} \binom{2n}{n} (n+1)^{k-1}x^{n+1}-\binom{2N}{N} N^{k}x^N \\
	&=2x+4x\sum_{n=1}^{N-1}\binom{2n}{n} (n+1)^{k}x^{n}-2x \sum_{n=1}^{N-1} \binom{2n}{n} (n+1)^{k-1}x^{n}-\binom{2N}{N} N^{k}x^N 
	\end{align*} 

\begin{align*}
\text{or}, \	S_k(N,x) 
	&=2x+4x\sum_{n=1}^{N-1}\binom{2n}{n} \sum_{u=0}^{k}\binom{k}{u} n^ux^{n}-2x \sum_{n=1}^{N-1} \binom{2n}{n} \sum_{v=0}^{k-1} \binom{k-1}{v}n^v x^{n}-\binom{2N}{N} N^{k}x^N \\
	&= 2x+4x\sum_{n=1}^{N-1} \binom{2n}{n} x^n+4x \sum_{u=1}^{k} \binom{k}{u}S_u(N,x)-2x \sum_{n=1}^{N-1} \binom{2n}{n}x^n \\ &-2x \sum_{v=1}^{k-1}\binom{k-1}{v} S_v(N,x) -\binom{2N}{N} N^{k}x^N \\
	&=2x+2x\sum_{n=1}^{N-1} \binom{2n}{n} x^n+4x \sum_{u=1}^{k} \binom{k}{u}S_u(N,x)-2x \sum_{v=1}^{k-1} \binom{k-1}{v}S_v(N,x)-\binom{2N}{N} N^{k}x^N, \\
	&=2xS_0(N,x)+4x \sum_{u=1}^{k} \binom{k}{u}S_u(N,x)-2x \sum_{v=1}^{k-1}\binom{k-1}{v} S_v(N,x)-\binom{2N}{N} N^{k}x^N, \\
	& \text{since} \ S_0(N,x)=\sum_{n=0}^{N-1} \binom{2n}{n} x^n. \\
\end{align*}
The recurrence formula is given by 
\begin{equation} \label{eq3}
S_k(N,x)=2xS_0(N,x)+4x \sum_{u=1}^{k} \binom{k}{u}S_u(N,x)-2x \sum_{v=1}^{k-1}\binom{k-1}{v} S_v(N,x)-\binom{2N}{N} N^{k}x^N, k \geq 1,
\end{equation} 
which will help us to deduce a summation formula. For, noting that $\binom{n}{r}=0$ if $n<r$, we have 

\begin{eqnarray}
\label{2.11}	S_1(N,x)&=&2xS_0(N,x)+4xS_1(N,x)- \binom{2N}{N}Nx^N \\
\label{2.12}	S_2(N,x)&=&2xS_0(N,x)+6xS_1(N,x)+4x S_2(N,x)- \binom{2N}{N} N^2x^N \\
\label{2.13}	S_3(N,x)&=&2xS_0(N,x)+8xS_1(N,x)+10xS_2(N,x)+4xS_3(N,x)- \binom{2N}{N}N^3x^N  \\
\label{2.14} S_4(N,x)&=& 2xS_0(N,x)+10xS_1(N,x)+18xS_2(N,x)+14xS_3(N,x)+4xS_4(N,x) \\ \nonumber && -\binom{2N}{N}N^4x^N  \\ \label{2.15} S_5(N,x)&=& 2xS_0(N,x)+12xS_1(N,x)+28xS_2(N,x)+40xS_3(N,x) +20xS_4(N,x) \\ &&+4xS_5(N,x)- \binom{2N}{N}N^5x^N \nonumber \\
\cdots \cdots \ \text{so on} \nonumber
\end{eqnarray}	

Thus the recurrence formula $\eqref{eq3}$ helps us to calculate the sum $S_{k}(N,x)$ provided all preceeding sums $S_i(N,x), \ i=1,2, \cdots, k$ as function of $S_0(N,x)$. Then in view of the relations $ \eqref{2.11}, \eqref{2.12}, \eqref{2.13}, \eqref{2.14}, \eqref{2.15} $ and by straightforward calculation, we obtain the following summation formula 
\begin{equation} \label{16}
\sum_{n=0}^{N-1}  \binom{2n}{n} \left[n^{k}(4x-1)^k+U_k(x) \right]x^n= \binom{2N}{N}x^{N}A_{k-1}(N,x),   
\end{equation}
which can be written as
\begin{equation} \label{eq4}
S_k(N,x)=-(4x-1)^{-k}U_k(x)S_0(N,x)+(4x-1)^{-k}  \binom{2N}{N}x^NA_{k-1}(N,x),   
\end{equation}
 where $U_k(x)$ is a polynomial in $x$ of degree $k$ and $A_{k-1}(N,x)$ is a polynomial in $x$ of degree $k-1$ and each coefficients in $A_{k-1}(N,x)$ are itself polynomials in $N$.
Using equation $\eqref{eq4}$ in $\eqref{eq3}$, we get recurrence formulas for $U_{k}(x)$ and $A_k(N,x)$:

\begin{align} \label{eq5}
	\sum_{u=1}^{k} 4x \binom{k}{u} (4x-1)^{k-u}U_u(x)-\sum_{v=1}^{k-1} 2x \binom{k-1}{v} (4x-1)^{k-v}U_v(x)-U_k(x)-2x(4x-1)^k=0
\end{align}
\begin{align} \label{eq6}
	\sum_{u=1}^{k} 4x \binom{k}{u} (4x-1)^{k-u} A_{u-1}(N,x)-\sum_{v=1}^{k-1} 2x \binom{k-1}{v} (4x-1)^{k-v}A_{v-1}(N,x)-A_{k-1}(N,x) \\
	 -N^k(4x-1)^k=0. \nonumber
\end{align}

The polynomial $A_k(N,x)$ has some assigned relation with $U_k(x)$.
If $N=0$, then equation $\eqref{eq6}$ takes the following form:
\begin{equation} \label{eq7}
\sum_{u=1}^{k} 4x \binom{k}{u}(4x-1)^{k-u} A_{u-1}(0,x)-\sum_{v=1}^{k-1} 2x \binom{k-1}{v} (4x-1)^{k-v} A_{v-1}(0,x)-A_{k-1}(0,x)=0.
\end{equation}
 For $N=1$, the equation $\eqref{eq6}$ takes the following form:
 \begin{equation} \label{eq8}
 \sum_{u=1}^{k} 4x \binom{k}{u}(4x-1)^{k-u} A_{u-1}(1,x)-\sum_{v=1}^{k-1} 2x \binom{k-1}{v} (4x-1)^{k-v} A_{v-1}(1,x)-A_{k-1}(1,x)-(4x-1)^k=0.
 \end{equation}
 Multiplying $\eqref{eq8}$ by $2x$ and subtracting $\eqref{eq7}$ from the result, we get
 \begin{align} \label{eq9}
 \sum_{u=1}^{k} 4x \binom{k}{u} \left[2xA_{u-1}(1,x)-A_{u-1}(0,x) \right]- \sum_{v=1}^{k-1}  2x\binom{k-1}{v} \left[2xA_{v-1}(1,x)-A_{v-1}(0,x) \right] \\-\left[ 2x A_{k-1}(1,x)-A_{k-1}(0,x)\right]-2x(4x-1)^k=0. \nonumber
 \end{align}
 Comparing the relations $\eqref{eq5}$ and $\eqref{eq9}$, we get the required relation
 \begin{equation}
 U_k(x)=2xA_{k-1}(1,x)-A_{k-1}(0,x), \ k \geq 1.
 \end{equation}
This completes the proof.
\end{proof}
We now state and prove the main result of our paper about the invariance of the p-adic sum of the p-adic functional power series $\eqref{eq1}$ using the above result.
\begin{thm} \label{thm2}
	The p-adic functional power series $\eqref{eq1}$ has the following p-adic invariant sum
	\begin{equation} \label{24}
	\sum_{n=0}^{\infty}  \binom{2n}{n} P_{k}(n,x)x^n=0
	\end{equation}
	provided 
	\begin{equation}
	P_{k}(n,x)=\sum_{j=1}^{k} B_j[n^{j}(4x-1)^j+U_j(x)], \ \forall x \in \mathbb{Q} \setminus \left\{\frac{1}{4} \right\},
	\end{equation}
	where $B_j \in \mathbb{Q}$.
\end{thm}
\begin{proof}
Taking $N \to \infty$ in $\eqref{16}$, the term $A_{k-1}(N,x)$ vanishes with respect to p-adic absolute value and gives the sum of following p-adic infinite functional series 
\begin{equation} \label{26}
\sum_{n=0}^{\infty}  \binom{2n}{n} \left[n^{k}(4x-1)^k+U_k(x) \right]x^n=0.
\end{equation}
Now taking $P_k(n,x)=\sum_{j=1}^{k} B_j[n^{j}(4x-1)^j+U_j(x)]$, the relation $\eqref{24}$ is obtained, for all $x \in \mathbb{Q} \setminus \{ \frac{1}{4}\}$ and the constants $B_j \in \mathbb{Q}$. This completes the proof.

\end{proof}

The equation $\eqref{26}$ is valid for any $k \in \mathbb{N}$ and has same form for any $k \in \mathbb{N}$.  Indeed, the equality is independent of p-adic properties for $x \in \mathbb{Q} \setminus \left\{\frac{1}{4} \right\}$. In other words, it is a p-adic invariant expression. It gives rational sum $0$ for all rational arguments. The polynomials $U_k(x)$ and $A_{k-1}(N,x)$ plays plays important roles throughout the work. We can investigate more deeply about the properties of $U_k(x)$ and $A_{k-1}(N,x)$ in the above results.
The recurrence formulas $\eqref{eq5}$ and $\eqref{eq6}$ are the tools to compute the polynomials $U_k(x)$ and $A_{k-1}(n,x)$.  \\ The first few of $U_k(x)$ and $A_{k-1}(n,x)$ are given below: \\
\hspace*{0cm} For $k=1$,
\[
\hspace*{-5.7cm}\left. \begin{cases}
	U_1(x)&=2x \\ A_0(n,x)&=n.
\end{cases}
\right.
\]
\hspace*{0cm}
For $k=2$,
\[
\hspace*{-3cm}\left. \begin{cases}
U_2(x)&=-2x-4x^2 \\ A_1(n,x)&=(4n^2-4n)x-n^2.
\end{cases}
\right.
\] 
\hspace*{0cm} For $k=3$,
\[
\hspace*{3cm}\left. \begin{cases}
U_3(x)&=8x^3+20x^2+2x \\ A_2(n,x)&=(16n^3-40n^2+72n)x^2-(8n^3-10n^2+8n)x+n^3.
\end{cases}
\right.
\]
\hspace*{0cm} For $k=4$,
\[
\left. \begin{cases}
U_4(x)&=-16x^4-144x^3-60x^2-2x \\ A_3(n,x)&=(64n^4-224n^3+272n^2-880n)x^3 \\ &-(48n^4-112n^3-4n^2-120n)x^2 \\
&+(12n^4-14n^3-18n^2-10n)x-n^4.
\end{cases}
\right.
\]
\\
\hspace*{0cm} For $k=5$,
\[
\hspace*{2.8cm}\left. \begin{cases}
U_5(x)&=800x^5+352x^4+1176x^3+136x^2+2x \\ A_4(n,x)&=(256n^5-1152n^4+1984n^3-1568n^2+7648n)x^4 \\
& -(256n^5-864n^4+480n^3+1216n^2-3008n)x^3 \\
& (96n^5-216n^4-132n^3+308n^2-108n)x^2 \\
& -(16n^5-18n^4-32n^3-28n^2-12n)x+n^5.
\end{cases}
\right.
\] 
\\ 
\hspace*{0cm} For $k=6$,
\[
\hspace*{4.3cm}\left. \begin{cases}
U_6(x)&=-4032x^6-3616x^5-32992x^4-2244x^3-332x^2-2x \\ A_5(n,x)&=(1024n^6-5632n^5+12544n^4-384n^3-2630n^2-147686n)x^5 \\
& -(128n^6-5632n^5+6208n^4-15684n^3+8688n^2+28304n)x^4 \\
&+(640n^6-2112n^5-19552n^4+6408n^3-384n^2-2792n)x^3 \\
&-(160n^6-352n^5-404n^4+444n^3+2436n^2-280n)x^2 \\
&(20n^6-22n^5-50n^4-60n^3-40n^2-14n)x-n^6. 
\end{cases}
\right.
\]
\newpage
\section{\textbf{Application of the Invariant Summation Formula}}
In this section we will connect the summation formula $\eqref{26}$ with Bernoulli numbers and Bernoulli polynomials.
\subsection{Connection with Bernoulli Numbers} 
 Let $ C^1( \mathbb{Z}_p \to \mathbb{C}_p)$  be the set of all p-adic continuous differentiable functions i.e, $$C^1( \mathbb{Z}_p \to \mathbb{C}_p)= \{f| \ f: \mathbb{ Z}_p \to \mathbb{C}_p, \ f(x) \ \text{ is differentiable and} \ \frac{d}{dx}f(x) \ \text{ is continuous} \}.$$ Then the Volkenborn integral \cite{WHS1} or p-adic Bosonic integral of the function $f \in C^1 (\mathbb{Z}_p \to \mathbb{C}_p)$ is given by
$$ \int_{\mathbb{Z}_p} x^n dx=\lim_{n \to \infty} \frac{1}{p^n} \sum_{x=0}^{p^n-1} f(x). $$ The above integral has the following property 
\begin{equation}
\int_{\mathbb{Z}_p} x^ndx=B_n, \ n= \mathbb{N} \cup \{0\},
\end{equation}
where $B_n$ is the $n^{th}$ Bernoulli number. The Bernoulli numbers satisfy the following recurrence relation 
$$ \sum_{i=0}^{m} \binom{m+1}{i}B_i=0, \ B_0=1, \ m=1,2,3, \cdots$$
Now we show the connection of the invariant summation formula $\eqref{26}$ with Bernoulli numbers as follows: \\
For $k=1$
\begin{eqnarray} \label{3.2}
&& \nonumber \hspace*{-2cm} \sum_{n=0}^{\infty} \binom{2n}{n} \left[n(4x-1)+2x \right]x^n =0, \ x \in \mathbb{Z}_p,   \\
 && \hspace*{-2cm} \sum_{n=0}^{\infty} \binom{2n}{n} \left[(4n+2)B_{n+1}-nB_n \right] =0.
\end{eqnarray}
For $k=2$
\begin{eqnarray} \label{3.3}
 \nonumber
&& \hspace*{0.6cm} \sum_{n=0}^{\infty} \binom{2n}{n} \left[n^2(4x-1)^2-(2x+4x^2) \right]x^n =0, \ x \in \mathbb{Z}_p,  \\
&& \hspace*{0.6cm} \sum_{n=0}^{\infty} \binom{2n}{n} \left[(16n^2+12)B_{n+2}-(8n^2-2)B_{n+1}-n^2B_n \right] =0.
\end{eqnarray}
\newpage
\hspace*{-0.5cm} For $k=3$
\begin{equation*} 
\hspace*{-1cm} \begin{split}
  & \sum_{n=0}^{\infty} \binom{2n}{n} [ n^3(4x-1)^3+(8x^3+20x^2+2x)]x^n =0, x \in \mathbb{Z}_p, 
  \end{split}
  \end{equation*}
  \begin{equation} \label{3.4}
 \hspace*{0.8cm} \begin{split}
 &\sum_{n=0}^{\infty} \binom{2n}{n} 
 \bigl[(64n^3+8)B_{n+3}-(48n^3-20)B_{n+2}+(12n^3+2)B_{n+1}-n^3B_n \bigr] =0.
\end{split}
\end{equation}
For $k=4$,
\begin{equation*}
\hspace*{-1cm} \begin{split}
&\sum \binom{2n}{n} \bigl[n^4(4x-1)^4 +(-16x^4-144x^3-60x^2-2x)\bigr]=0,
\end{split}
\end{equation*}
\begin{equation} \label{3.5}
 \hspace*{0.4cm}\begin{split}
		&\sum \binom{2n}{n} \bigl[(256n^4-16) B_{n+4}-(256n^4+144)B_{n+3}+(96n^4-60)B_{n+2} \\ & -(16n^4+2)B_{n+1}+n^{4}B_n\bigr]=0.
	\end{split}
\end{equation}
For $k=5$,
\begin{equation*}
\hspace*{0.4cm}	\begin{split}
		&\sum \binom{2n}{n} \bigl[n^5(4x-1)^5 +(800x^5+352x^4+1176x^3+136x^2+2x)\bigr]=0,
	\end{split}
\end{equation*}
\begin{equation} \label{3.6}
\hspace*{0.6cm} \begin{split}
&\sum \binom{2n}{n} \bigl[(1024n^5+800)B_{n+5}-(1280n^5-352)B_{n+4}+(640n^5-1176)B_{n+3} \\ &-(160n^5+136)B_{n+2}
+(20n^5+2)B_{n+1}-n^5B_n\bigr]=0.
\end{split}
\end{equation}
For $k=6$,
 \begin{equation*}
\hspace*{1.8cm}	\begin{split}
		&\sum \binom{2n}{n} \bigl[n^6(4x-1)^6 +(-4032x^6-3616x^5-32992x^4-2244x^3-332x^2-2x)\bigr]=0,
	\end{split}
\end{equation*}
\begin{equation} \label{3.7}
\hspace*{1cm} \begin{split}
&\sum \binom{2n}{n} \bigl[(4096n^6-4032)B_{n+6}-(6144n^6+3616)B_{n+5}+(3840n^6-32992)B_{n+4} \\
&-(1280n^6+2244)B_{n+3}+(240n^6-332)B_{n+2}-(24n^6+2)B_{n+1}+n^6B_n \bigr]=0.
\end{split}
\end{equation}

The important fact to notice is that all the above series $ \eqref{3.2}, ~ \eqref{3.3}, ~ \eqref{3.4}, ~ \eqref{3.5}, ~ \eqref{3.6}, ~ \eqref{3.7}$ of the Bernoulli numbers are p-adic convergent because $|B_n|_p \leq p$, \cite{WHS1}. 

\subsection{Connection with Bernoulli Polynomials :} Bernoulli polynomials satisfy the following relation \begin{equation}
B_n(x+1)-B_n(x)=nx^{n-1}.
\end{equation} 
Then the summation formula $\eqref{26}$ produces the following relations with Bernoulli polynomials: \\
For $k=1$,
\begin{equation} \label{3.9}
\begin{split}
&\sum \binom{2n}{n} \Biggl[\frac{4n+2}{n+2}  \Bigl\{B_{n+2}(x+1)-B_{n+2}(x) \Bigr\}-\frac{n}{n+1} \Bigl\{B_{n+1}(x+1)-B_{n+1}(x) \Bigr\}\Biggr]=0,
\end{split}
\end{equation}
For $k=2$,
\begin{equation} \label{3.10}
\hspace*{-0.5cm} \begin{split}
& \sum \binom{2n}{n} \Biggl[\frac{16n^2-4}{n+3}  \Bigl\{B_{n+3}(x+1)-B_{n+3}(x) \Bigr\}-\frac{8n^2+2}{n+2}  \Bigl\{B_{n+2}(x+1)-B_{n+2}(x) \Bigr\} \\
 & + \frac{n^2}{n+1} \Bigl\{B_{n+1}(x+1)-B_{n+1}(x) \Bigr\} \Biggr]=0,
  \end{split}
\end{equation}
For $k=3$,
\begin{equation} \label{3.11}
\hspace*{0.1cm} \begin{split}
& \sum \binom{2n}{n} \Biggl[\frac{64n^3+8}{n+4}  \Bigl\{B_{n+4}(x+1)-B_{n+4}(x) \Bigr\}-\frac{48n^3-20}{n+3}  \Bigl\{B_{n+3}(x+1)-B_{n+3}(x) \Bigr\} \\
&+ \frac{12n^3+2}{n+2}  \Bigl\{B_{n+2}(x+1)-B_{n+2}(x) \Bigr\}-\frac{n^3}{n+1}  \Bigl\{B_{n+1}(x+1)-B_{n+1}(x) \Bigr\} \Biggr]=0,
\end{split}
\end{equation}

For $k>3$, there are similar relations involving Bernoulli polynomials. Putting $x=0$ in the relations $\eqref{3.9}$, $\eqref{3.10}$ and $\eqref{3.11}$, we will come back to the relations involving Bernoulli numbers. In other words, putting $x=0$ in the relations $\eqref{3.9}$, $\eqref{3.10}$ and $\eqref{3.11}$, we will get some p-adic convergent series involving Bernoulli numbers. 
\section{\textbf{Conclusion}}
The paper deals with the p-adic convergence of the constructed p-adic functional series with binomial coefficient. We have then deduced a summation formula of p-adic power series and shown that the summation formula is an invariant, i.e., for rational argument, the sum is rational and in fact the invariant sum is $0$. Then we make a direct application of the invariant summation formula with Bernoulli polynomials and Bernoulli numbers producing some relations involving Bernoulli numbers and Bernoulli polynomials.  

\textbf{Acknowledgement:} 
The second author is grateful to The Council Of Scientific and Industrial Research (CSIR), Government of India, for the award of JRF (Junior Research Fellowship).

\end{document}